\documentclass[11pt, letterpaper]{amsart}
\usepackage[arrow,cmtip,matrix]{xy}
\usepackage{amsmath,amsfonts,amssymb,url,verbatim}
\usepackage[utf8]{inputenc}
\usepackage[T1]{fontenc}
\usepackage{latexsym}
\usepackage[pagebackref]{hyperref}
\usepackage{cleveref}

\newcommand\citestacks[1]{\cite[\href{https://stacks.math.columbia.edu/tag/#1}{tag #1}]{Stacks}}

\newcommand{\A}{\mathbb{A}}
\newcommand{\C}{\mathbb{C}}
\newcommand{\E}{\mathcal{E}}

\newcommand{\II}{\mathcal{I}}

\newcommand{\OO}{\mathcal{O}}
\renewcommand{\P}{\mathbb{P}}

\newcommand{\Q}{\mathbb{Q}}
\newcommand{\R}{\mathbb{R}}
\newcommand{\Z}{\mathbb{Z}}

\newcommand{\fa}{\mathfrak{a}}
\newcommand{\fb}{\mathfrak{b}}
\newcommand{\fc}{\mathfrak{c}}
\newcommand{\fd}{\mathfrak{d}}
\newcommand{\fm}{\mathfrak{m}}

\newcommand{\Gm}{\mathbb{G}_{\mathrm{m}}}
\newcommand{\tors}{\mathrm{tor}}
\newcommand{\doublebackslash}{\mathbin{\backslash\!\backslash}}
\newcommand{\isom}{\stackrel{\sim}{\longrightarrow}}

\DeclareMathOperator{\GL}{GL}
\DeclareMathOperator{\Gal}{Gal}
\DeclareMathOperator{\Norm}{N}  
\DeclareMathOperator{\SL}{SL}
\DeclareMathOperator{\id}{id}

\newcommand\degred{\deg_{\mathrm{red}}}

\newtheorem{tm}{Theorem}[section]

\newtheorem{lemma}[tm]{Lemma}
\newtheorem{corollary}[tm]{Corollary}

\theoremstyle{definition}
\newtheorem{definition}[tm]{Definition}

\newtheorem{remark}[tm]{Remark}
\newtheorem{example}[tm]{Example}

\title[Counting elliptic curves with level structures]{Counting elliptic curves with prescribed level structures over number fields}
\author{Peter Bruin}
\address{Mathematisch Instituut, Universiteit Leiden, Postbus 9512, 2300 RA Leiden, Netherlands}
\email{P.J.Bruin@math.leidenuniv.nl}
\urladdr{http://pub.math.leidenuniv.nl/~bruinpj/}
\author{Filip Najman}
\address{University of Zagreb, Faculty of Science, Department of Mathematics,  Bijeni\v{c}ka cesta 30, 10000 Zagreb, Croatia}
\email{fnajman@math.hr}
\urladdr{http://web.math.pmf.unizg.hr/~fnajman/}
\thanks{This work was supported by the QuantiXLie Centre of Excellence, a project
co-financed by the Croatian Government and European Union through the
European Regional Development Fund - the Competitiveness and Cohesion
Operational Programme (Grant KK.01.1.1.01.0004) and by the Croatian Science Foundation under the project no.\ IP-2018-01-1313.
The first-named author was partially supported by the Dutch Research Council (NWO/OCW), as part of the Quantum Software Consortium programme (project number 024.003.037).
This project began during a joint stay of the authors at the Max-Planck-Institut f\"ur Mathematik, Bonn, in October 2018.  We are grateful to MPIM for its funding and hospitality.}

\begin{document}

\maketitle

\begin{abstract}
Harron and Snowden \cite{HS} counted the number of elliptic curves over~$\Q$ up to height~$X$ with torsion group $G$ for each possible torsion group~$G$ over~$\Q$. In this paper we generalize their result to all number fields and all level structures~$G$ such that the corresponding modular curve $X_G$ is a weighted projective line $\P(w_0,w_1)$ and the morphism $X_G\to X(1)$ satisfies a certain condition.  In particular, this includes all modular curves $X_1(m,n)$ with coarse moduli space of genus~$0$.
We prove our results by defining a \emph{size function} on $\P(w_0,w_1)$ following unpublished work of Deng~\cite{Deng}, and working out how to count the number of points on $\P(w_0,w_1)$ up to size~$X$.
\end{abstract}

\section{Introduction}
Let $E$ be an elliptic curve over a number field~$K$. The Mordell--Weil theorem says that $E(K)$ is isomorphic to $\Z^r \times E(K)_\tors$ for some $r\ge0$, where $E(K)_\tors$ is the (finite) torsion subgroup of $E(K)$. It is a natural question which groups appear as $E(K)_\tors$, and moreover how often each one of these groups appears.
Harron and Snowden \cite{HS} studied this question and answered it in the case $K=\Q$. The aim of this paper is to study the same problem, but to both allow $K$ to be any number field and to answer the more general question how often a prescribed $G$-level structure appears.

To make this question more precise, let $n$ be a positive integer, let $G$ be a subgroup of $\GL_2(\Z/n\Z)$, and let $K$ be a number field.  We say that an elliptic curve $E$ over $K$ \emph{admits a $G$-level structure} if there exists a $(\Z/n\Z)$-basis of $E[n](\bar K)$ such that the Galois representation
$\rho_{E,n}\colon \Gal(\bar K/K)\to \GL_2(\Z/n\Z)$
defined by this basis has image contained in $G$.  We write
$$
\E_{G,K} = \{\text{elliptic curves over~$K$ admitting a $G$-level structure}\}/{\cong}.
$$

We will define a \emph{size function} $S_K$ from the set of isomorphism classes of elliptic curves over~$K$ to $\R_{>0}$; see Definition~\ref{def:size-ell}.  We define a function $N_{G,K}\colon\R_{>0}\to\Z_{\ge0}$ by
$$
N_{G,K}(X) = \#\{E\in \E_{G,K}\mid S_K(E)^{12}\le X\}.
$$

Let $X_G$ be the moduli stack of generalized elliptic curves with $G$-level structure.  This is a one-dimensional proper smooth geometrically connected algebraic stack over the fixed field $K_G$ of the action of~$G$ on $\Q(\zeta_n)$ given by $(g,\zeta_n)\mapsto\zeta_n^{\det g}$.
We consider cases where $X_G$ is a weighted projective line $\P(w_0,w_1)$ over~$K_G$.
We can now state our main result (which is also stated in a slightly different form in \Cref{thm:final}).

\begin{tm}
\label{thm:intro}
Let $n$ be a positive integer, and let $G$ be a subgroup of\/ $\GL_2(\Z/n\Z)$.
Assume that the stack~$X_G$ over~$K_G$ is isomorphic to $\P(w)_{K_G}$, where $w=(w_0,w_1)$ is a pair of positive integers, and let $e(G)$ be the reduced degree of the canonical morphism $X_G\to X(1)$ (see Definition~\ref{def:reduced-degree}).
Furthermore, assume $e(G)=1$ or $w=(1,1)$ holds.
Then for every finite extension $K$ of~$K_G$, we have
$$
N_{G,K}(X) \asymp X^{1/d(G)}
\quad\text{as }X\to\infty,
$$
where
$$
d(G) = \frac{12e(G)}{w_0+w_1}.
$$
\end{tm}

As all modular curves $X_G=X_1(m,n)$ with coarse moduli space of genus~$0$ satisfy the assumptions of \Cref{thm:intro}, our result generalizes \cite[Theorem~1.2]{HS}, where this statement was proved in the case where $K=\Q$ and where $G$ is one of the 15 groups corresponding to the torsion groups from Mazur's theorem.

A recent result of Pizzo, Pomerance and Voight \cite{PPV} is $N_{G,\Q}(X)\sim X^{1/2}$ for $G$ such that $X_G=X_0(3)$.  Moreover, they determined the constant in front of the leading term of the function $N_{G,\Q}(X)$ as well as the first two lower-order terms.  This result falls outside of the reach of our results, as $X_0(3)$ is not a weighted projective line (cf.\ Remark~\ref{remark:counterexample}).

Similarly, Pomerance and Schaefer \cite{PS} proved that $N_{G,\Q}(X)\sim X^{1/3}$ for $G$ such that $X_G=X_0(4)$, and determined the constants in front of the leading term and the first lower-order term.  Our result implies $N_{G,K}\asymp X^{1/3}$ for all number fields $K$; for $K=\Q$, this follows from the sharper results of~\cite{PS}.

Cullinan, Kenney and Voight \cite[Theorem~1.3.3]{CKV} proved a sharper version of \Cref{thm:intro} in the special case where $X_G$ is a projective line (i.e.\ isomorphic to $\P^1=\P(1,1)$) and $K=\Q$.  More precisely, they give an asymptotic expression for $N_{G,\Q}(X)$ containing an effectively computable leading coefficient and an error term.

Boggess and Sankar \cite{Boggess-Sankar} determined the growth rate of the number of elliptic curves over~$\Q$ with a cyclic $n$-isogeny for $n\in\{2,3,4,5,6,8,9,12,16,18\}$.  Out of these, only the cases $n=2$ and $n=4$ (for which a more precise result was already proved in \cite{HS,PS}) correspond to weighted projective lines and are therefore generalized to number fields by \Cref{thm:intro}.

\begin{remark}
The 12-th power in the definition of $N_{G,K}(X)$ is included for easier comparison with the height function in \cite{HS}; see Remark~\ref{remark:size-comparison}.
\end{remark}

\begin{remark}
Our result gives a more conceptual interpretation of $d(G)$; cf.\ \cite[\S1.2]{HS}.  Namely, we show that $d(G)$ can be expressed in terms of the pair of positive integers $(w_0,w_1)$ for which $X_G$ is isomorphic to the weighted projective line with weights $(w_0,w_1)$, and~$e(G)$, an invariant (similar to the degree) of the morphism $X_G\to X(1)$.

We also remark that our result shows how in certain cases one can count points in the image of a morphism of stacks, partially answering a question in \cite[Remark~1.5]{HS}.
\end{remark}

\section{Weighted projective spaces}

\begin{definition}
Given an $(n+1)$-tuple $w=(w_0,\ldots,w_n)$ of positive integers, the \emph{weighted projective space with weights $w$} is the algebraic stack
$$
\P(w) = [\Gm\backslash\A^{n+1}_{\ne0}]
$$
over $\Z$, where $\A^{n+1}_{\ne0}$ is the complement of the zero section in~$\A^{n+1}$ and $\Gm$ acts on $\A^{n+1}_{\ne0}$ by
$$
(\lambda,(x_0,\ldots,x_n))\longmapsto(\lambda^{w_0}x_0,\ldots,\lambda^{w_n}x_n).
$$
\end{definition}

It is known that $\P(w)$ is a proper smooth algebraic stack over~$\Z$, and in fact a complete smooth toric Deligne--Mumford stack in the sense of Fantechi, Mann and Nironi \cite[Example 7.27]{FMN}.
For every ring~$R$, there is a \emph{groupoid of $R$-points} of~$\P(w)$.  We will mostly be interested in the set of isomorphism classes of this groupoid, which we call the \emph{set of $R$-points} of $\P(w)$ and denote by $\P(w)(R)$.  Given a field~$L$, the set $\P(w)(L)$ can be described as
$$
\P(w)(L)=L^\times\backslash(L^{n+1}\setminus\{0\}),
$$
where $L^\times$ acts on $L^{n+1}\setminus\{0\}$ by
$$
(\lambda,(x_0,\ldots,x_n))\longmapsto(\lambda^{w_0}x_0,\ldots,\lambda^{w_n}x_n).
$$
The image in $\P(w)(L)$ of an element $x\in L^{n+1}\setminus\{0\}$ will be denoted by $[x]$.

\begin{example}
If $w=(m)$ with $m$ a positive integer, then $\P(m)$ is canonically isomorphic to the classifying stack of the group scheme~$\mu_m$.  If $L$ is a field, then we have
$$
\P(m)(L) = (L^\times)^m\backslash L^\times.
$$
\end{example}

\section{Size functions}

Let $w$ be an $(n+1)$-tuple as above, let $K$ be a number field, and let $\OO_K$ be its ring of integers.  On the set $\P(w)(K)$, we define a \emph{size function} similarly to Deng \cite{Deng}; see Definition~\ref{def:size} below.  We do not restrict to weighted projective spaces that are ``well-formed'' in the sense of~\cite{Deng}.

\begin{definition}
For $x\in K^{n+1}$, the \emph{scaling ideal} of~$x$, denoted by\/ $\II_w(x)$, is the intersection of all fractional ideals $\fa$ of~$\OO_K$ satisfying $x\in\fa^{w_0}\times\cdots\times\fa^{w_n}$.  Similarly, for an $(n+1)$-tuple $(\fb_0,\ldots,\fb_n)$ of fractional ideals of~$\OO_K$, the \emph{scaling ideal} of $(\fb_0,\ldots,\fb_n)$, denoted by\/ $\II_w(\fb_0,\ldots,\fb_n)$, is the intersection of all fractional ideals $\fa$ of~$\OO_K$ satisfying $\fb_i\subseteq\fa^{w_i}$ for all $i$.
\end{definition}

\begin{remark}
For all $x\in K^{n+1}\setminus\{0\}$, the fractional ideal $\II_w(x)$ is non-zero and satisfies
$$
\II_w(x)^{-1} = \{a\in K\mid a^{w_i} x_i\in\OO_K\text{ for }i=0,\ldots,n\}.
$$
Similarly, for every $(n+1)$-tuple $(\fb_0,\ldots,\fb_n)$ of fractional ideals of~$\OO_K$, not all zero, the fractional ideal $\II_w(\fb_0,\ldots,\fb_n)$ is non-zero and satisfies
$$
\II_w(\fb_0,\ldots,\fb_n)^{-1} = \{a\in K\mid a^{w_i}\fb_i\subseteq\OO_K\text{ for }i=0,\ldots,n\}.
$$
\end{remark}

\begin{definition}
Let $\Omega_{K,\infty}$ denote the set of Archimedean places of~$K$, and for each $v\in\Omega_{K,\infty}$, let $|\ |_v\colon K\to\R_{\ge0}$ be the corresponding normalized absolute value.  The \emph{Archimedean size} on $K^{n+1}\setminus\{0\}$ is the function
$$
\begin{aligned}
H_{w,\infty}\colon K^{n+1}\setminus\{0\} &\longrightarrow \R_{>0}\\
x &\longmapsto \prod_{v\in\Omega_{K,\infty}} \max_{0\le i\le n}|x_i|_v^{1/w_i}.
\end{aligned}
$$
\end{definition}

\begin{definition}
\label{def:size}
The \emph{size function} on $\P(w)(K)$ is the function
$$
\begin{aligned}
S_{w,K}\colon \P(w)(K)&\longrightarrow\R_{>0}\\
[x]&\longmapsto\Norm(\II_w(x))^{-1}H_{w,\infty}(x).
\end{aligned}
$$
\end{definition}

It is straightforward to check that $S_{w,K}([x])$ does not depend on the choice of the representative~$x$.

\begin{example}
If $w=(m)$ with $m$ a positive integer and $x\in\Z\setminus\{0\}$ is $m$-th power free, we have
$$
S_{(m),\Q}([x]) = |x|^{1/m}.
$$
\end{example}

\begin{remark}
If $L/K$ is an extension of number fields, we have
$$
S_{(1,\ldots,1),L}(x) = S_{(1,\ldots,1),K}(x)^{[L:K]},
$$
but for general weights $w$ such a relation does not hold.  For example, if $w=(m)$ with $m\ge2$ and $x\in\Z\setminus\{0\}$ is $m$-th power free, then
$$
S_{(m),\Q}([x])=|x|^{1/m},
$$
but
$$
S_{(m),\Q(x^{1/m})}([x])=S_{(m),\Q(x^{1/m})}([1])=1.
$$
\end{remark}

\begin{remark}
Definition~\ref{def:size} is a special case of the notion of height for rational points on algebraic stacks defined by Ellenberg, Satriano and Zureick-Brown \cite{ESZ}.  Namely, as explained in \cite[Section~3.3]{ESZ}, we have
$$
\log S_{w,\Q}(x) = \mathrm{ht}_{\mathcal{L}}(x),
$$
where $\mathrm{ht}_{\mathcal{L}}$ is the height function corresponding to the tautological line bundle $\mathcal{L}$ on~$\P(w)$.
The work of Ellenberg, Satriano and Zureick-Brown was recently used by Boggess and Sankar \cite{Boggess-Sankar} to count elliptic curves over~$\Q$ with a rational $n$-isogeny for $n\in\{2,3,4,5,6,8,9\}$, as mentioned in the introduction.
\end{remark}

\begin{tm}
\label{thm:count}
Let $n$ be a non-negative integer, let $w=(w_0,\ldots,w_n)$ be an $(n+1)$-tuple of positive integers, and let $K$ be a number field.  Let $r_1$, $r_2$, $d_K$, $h_K$, $R_K$, $\mu_K$ and $\zeta_K$ be the number of real places, number of non-real complex places, discriminant, class number, regulator, number of roots of unity and Dedekind $\zeta$-function of~$K$, respectively.
We write
$$
\begin{aligned}
|w|&=w_0+w_1+\cdots+w_n,\\
\mu_K^w&=\frac{\mu_K}{\gcd\{w_0,w_1,\ldots,w_n,\mu_K\}}\\
C_K^w &= \frac{h_K R_K}{\mu_K^w\zeta_K(|w|)}
\biggl(\frac{2^{r_1}(2\pi)^{r_2}}{\sqrt{|d_K|}}\biggr)^{n+1}
|w|^{r_1+r_2-1}.
\end{aligned}
$$
Then we have
$$
\#\{x\in\P(w)(K)\mid S_{w,K}(x)\le T\}\sim C_K^w T^{|w|}
\quad\text{as }T\to\infty.
$$
\end{tm}

\begin{proof}
This was proved by Deng \cite[Theorem~(A)]{Deng} in the case where $\P(w)$ is \emph{well-formed}, i.e.\ each $n$ elements from~$w$ are coprime.  However, the proof works in general with only minor changes: in the paragraph before \cite[Proposition 4.2]{Deng}, the statement that the group of roots of unity acts effectively has to be replaced by the statement that all orbits of points with all coordinates non-zero contain $\mu_K^w$ points, and the factor $w$ (denoting the number of roots of unity) in \cite[Proposition~4.2, Proposition~4.5, Corollary~4.6 and Theorem~(A)]{Deng} has to be replaced by $\mu_K^w$.
\end{proof}

\begin{remark}
Theorem~\ref{thm:count} also follows from recent results of Darda \cite{Darda} on counting rational points on weighted projective spaces with respect to more general height functions; see in particular \cite[Remark 7.3.2.5]{Darda}.
\end{remark}

In the remainder of this article, we will only consider \emph{weighted projective lines}, i.e.\ one-dimensional weighted projective spaces where the weight is given by a pair $(w_0,w_1)$.

\section{Morphisms between weighted projective lines}

Let $u=(u_0,u_1)$, $w=(w_0,w_1)$ be two pairs of positive integers.  In this section, we classify the morphisms of stacks from $\P(w)$ to $\P(u)$ over a field.  These morphisms form a groupoid, but for simplicity we will only be interested in the set of isomorphism classes of this groupoid, or in other words the \emph{set of morphisms} from $\P(w)$ to $\P(u)$.
We also prove some facts about such morphisms generalizing the corresponding facts about morphisms $\P^1\to\P^1$.

\begin{lemma}
\label{lemma:morphism}
Let $K$ be a field, and let $u=(u_0,u_1)$, $w=(w_0,w_1)$ be two pairs of positive integers.  We consider $R=K[x_0,x_1]$ as a graded $K$-algebra where $x_0$ and~$x_1$ are homogeneous of degrees $w_0$ and~$w_1$, respectively.  Let $P_{u,w}(K)$ be the set of pairs $(f_0,f_1)\in R\times R$ with the following properties:
\begin{enumerate}
\item There exists $e=e(f_0,f_1)\in\Z_{\ge0}$ for which $f_0$ and~$f_1$ are homogeneous of degrees $eu_0$ and $eu_1$, respectively.
\item The homogeneous ideal $\sqrt{(f_0,f_1)}\subseteq R$ contains $(x_0,x_1)$.
\end{enumerate}
Let $K^\times$ act on $P_{u,w}(K)$ by $c(f_0,f_1)=(c^{u_0}f_0,c^{u_1}f_1)$.
Then there is a natural bijection from $K^\times\backslash P_{u,w}(K)$ to the set of morphisms $\P(w)_K\to\P(u)_K$ sending the class of $(f_0,f_1)\in P_{u,w}(K)$ to the morphism induced by the $K$-algebra homomorphism
$$
\begin{aligned}
K[y_0,y_1]&\longrightarrow K[x_0,x_1]\\
y_0&\longmapsto f_0\\
y_1&\longmapsto f_1.
\end{aligned}
$$
\end{lemma}

\begin{proof}
We apply Lemma~\ref{lemma:morphism-final} to the following data over~$K$:
\begin{itemize}
\item $X=\A^2_{\ne0}$ with coordinates $x=(x_0,x_1)$,
\item $Y=\A^2_{\ne0}$ with coordinates $y=(y_0,y_1)$,
\item $G=\Gm$ with coordinate~$g$,
\item $H=\Gm$ with coordinate~$h$,
\item $a\colon G\times X\to X$ is the weight~$w$ action, given on points by $a(g,x)=(g^{w_0}x_0,g^{w_1}x_1)$,
\item $b\colon H\times Y\to Y$ is the weight~$u$ action, given on points by $b(h,y)=(h^{u_0}y_0,h^{u_1}y_1)$.
\end{itemize}
(Note that the lemma applies because the Picard group of~$X$ is trivial.)

We first determine the morphisms $h\colon G\times X\to H$ satisfying the ``cocycle condition'' \eqref{eq:cocycle} of Lemma~\ref{lemma:morphism-final}.  A morphism $h\colon G\times X\to H$ is given by a monomial of the form $h(g,x)=\lambda g^e$ with $\lambda\in K^\times$ and $e\in\Z$, and $h$ satisfies \eqref{eq:cocycle} if and only if $\lambda=1$, i.e.\ $h$ is of the form $h(g,x)=g^e$.

Given $h$ as above, we now determine the morphisms $f\colon X\to Y$ such that the pair $(f,h)$ satisfies condition~\eqref{eq:compat} of Lemma~\ref{lemma:morphism-final}.  Every such $f$ is given by a pair $(f_0,f_1)\in R\times R$, and $(f_0,f_1)$ determines a morphism~$X\to Y$ if and only if $\sqrt{(f_0,f_1)}$ contains $(x_0,x_1)$.  It is straightforward to check that condition~\eqref{eq:compat} translates to the condition that $f_j$ is homogeneous of degree~$eu_j$ for $j=0,1$.  In particular, morphisms $f\colon X\to Y$ such that $(f,h)$ defines a morphism $[G\backslash X]\to [H\backslash Y]$ only exist if $e\ge0$; moreover, $e$ and therefore~$h$ are uniquely determined by~$f$.

Finally, the group $H(X)$ is canonically isomorphic to~$K^\times$, and if $(f,h)$ is a pair as above where $f$ is defined by $(f_0,f_1)$, and $c\in H(X)$, then we have $c(f,h)=(f',h)$ where $f'$ is defined by $(c^{u_0}f_0,c^{u_1}f_1)$.  The lemma therefore follows from Lemma~\ref{lemma:morphism-final}.
\end{proof}

\begin{definition}
\label{def:reduced-degree}
Let $K$ be a field, let $u$, $w$ be two pairs of positive integers, and let $\phi\colon\P(w)_K\to\P(u)_K$ be a morphism.  The \emph{reduced degree} of~$\phi$, denoted by\/ $\degred\phi$, is the unique integer~$e\ge0$ satisfying Lemma~\ref{lemma:morphism}(i) for some (hence every) pair $(f_0,f_1)$ giving rise to~$\phi$ via the bijection of Lemma~\ref{lemma:morphism}.
\end{definition}

\begin{remark}
\label{remark:representable}
A morphism $\phi\colon\P(u)_K\to\P(w)_K$ is representable (by which we mean representable in algebraic spaces) if and only if $\phi$ is faithful as a functor \citestacks{04Y5}.  Moreover, it suffices to check this condition on geometric fibres \cite[Corollary 2.2.7]{Conrad}.  From this one can deduce that $\phi$ is representable if and only if its reduced degree $e=\degred\phi$ satisfies
$$
\gcd(w_0,e)=\gcd(w_1,e)=1.
$$
\end{remark}

\begin{lemma}
\label{lemma:Nakayama}
In the setting of Lemma~\ref{lemma:morphism}, let $(f_0,f_1)\in P_{u,w}(K)$ and assume $e(f_0,f_1)>0$.  Then $R$ is finite over its graded subalgebra $S=K[f_0,f_1]$.
\end{lemma}

\begin{proof}
We write $R_+=Rx_0+Rx_1$, $S_+=Sf_0+Sf_1$ and $I=Rf_0+Rf_1=RS_+$.  By condition (ii) of Lemma~\ref{lemma:morphism} and the fact that $f_0$ and~$f_1$ are non-constant, we have $\sqrt{I}=R_+$.  Hence for $m$ sufficiently large, we have $R_+^m\subseteq I$, so the graded $K$-algebra $R/I$ is a quotient of $R/R_+^m$ and is therefore finite-dimensional over~$K$.  Choose homogeneous elements $g_1,\ldots,g_r\in R$ such that their images in $R/I$ are a $K$-basis of $R/I$.  In particular, the $g_i$ generate $R/I=R/RS_+$ over~$S$, so we have
$$
R=RS_+ + Sg_1 + \cdots + Sg_r.
$$
Hence the $\Z_{\ge0}$-graded $S$-module $M=R/(Sg_1+\cdots+Sg_r)$ satisfies $S_+M=M$.  It follows from a variant of Nakayama's lemma (see for example Eisenbud \cite[Exercise 4.6]{Eisenbud}) that $M=0$ and hence $R=Sg_1+\cdots+Sg_r$.
\end{proof}

\begin{lemma}
\label{lemma:finite}
Let $K$ be a field, let $u$, $w$ be two pairs of positive integers, and let $\phi\colon\P(w)_K\to\P(u)_K$ be a non-constant representable morphism.  Then $\phi$ is finite.
\end{lemma}

\begin{proof}
Since $\P(w)_K$ and $\P(u)_K$ are Deligne--Mumford stacks and $\phi$ is representable, we may choose a Cartesian diagram
$$
\xymatrix{
T\ar[d]\ar[r]^{\phi'}& S\ar[d]\\
\P(w)_K\ar[r]^\phi& \P(u)_K}
$$
where $S$ and~$T$ are algebraic spaces and the vertical maps are \'etale coverings.  Then $\phi'$ is proper because $\phi$ is proper, and is locally quasi-finite because $\phi'$ has relative dimension~$0$ \citestacks{04NV}.  In particular, $\phi'$ is representable in schemes \citestacks{0418} and is finite \citestacks{0A4X}.  It follows that $\phi$ is finite.
\end{proof}

\begin{remark}
Alternatively, Lemma~\ref{lemma:finite} may be proved using Lemma~\ref{lemma:Nakayama}.
\end{remark}

\begin{corollary}
\label{cor:normalization}
With the notation of Lemma~\ref{lemma:finite}, let $V\subseteq\P(w)_K$ be a dense open substack.  Then $\P(w)_K$ is the integral closure of\/ $\P(u)_K$ in $V$.
\end{corollary}

\begin{proof}
By Lemma~\ref{lemma:finite}, the morphism $\phi$ is finite and in particular integral.  Furthermore, $\P(w)_K$ is normal because $K[x_0,x_1]$ is integrally closed.  This proves the claim.
\end{proof}

\section{Some results on scaling ideals}

Let $K$ be a number field.  We prove two elementary results about scaling ideals.

\begin{lemma}
\label{lemma:poly}
Let $w=(w_0,w_1)$ be a pair of positive integers.  We consider $K[x_0,x_1]$ as a graded $K$-algebra by assigning weight $w_i$ to~$x_i$.  Let $f\in K[x_0,x_1]$ be homogeneous of degree~$d$.  Let $\fa(f)$ be the fractional ideal generated by the coefficients of~$f$.  Then for all $z\in K^2$, we have
$$
f(z) \in \fa(f)\II_w(z)^d.
$$
\end{lemma}

\begin{proof}
We abbreviate
$$
\fm = \II_w(z),
$$
so we have $z_0\in\fm^{w_0}$ and $z_1\in\fm^{w_1}$.  We write
$$
f = \sum_{k_0,k_1} a_{k_0,k_1} x_0^{k_0}x_1^{k_1}
$$
where the sum ranges over all pairs $(k_0,k_1)$ of non-negative integers such that $k_0w_0+k_1w_1=d$, and $a_{k_0,k_1}\in K$.  We now compute
\begin{align*}
f(z_0,z_1) &= \sum_{k_0,k_1}a_{k_0,k_1}z_0^{k_0}z_1^{k_1}\\
&\in \sum_{k_0,k_1}a_{k_0,k_1}(\fm^{w_0})^{k_0}(\fm^{w_1})^{k_1}\\
&=\sum_{k_0,k_1}a_{k_0,k_1}\fm^d\\
&=\fa(f)\fm^d,
\end{align*}
which proves the claim.
\end{proof}

\begin{lemma}
\label{lemma:scaling-root}
Let $z\in K$, and let
$$
h=x^d+c_1x^{d-1}+\cdots+c_d\in K[x]
$$
be a monic polynomial such that $h(z)=0$.  Suppose $\fb_1,\ldots,\fb_d$ are fractional ideals of~$K$ such that $c_i\in\fb_i$ for all $i$.  Then we have
$$
z\in\II_{(1,\ldots,d)}(\fb_1,\ldots,\fb_d).
$$
\end{lemma}

\begin{proof}
If all the $\fb_i$ are zero, then $z$ vanishes and the claim is trivial.  Now assume not all of the $\fb_i$ are zero.  We write
$$
\mathfrak{a} = \II_{(1,\ldots,d)}(\fb_1,\ldots,\fb_d)^{-1}
= \{a\in K\mid a\fb_1,a^2\fb_2,\ldots,a^d\fb_d\subseteq\OO_K\}.
$$
Then for all $a\in\mathfrak{a}$ we have
$$
0=a^dh(z)=(az)^d+(ac_1)(az)^{d-1}+\cdots+(a^dc_d).
$$
By assumption, each $a^ic_i$ lies in $a^i\fb_i$ and hence in~$\OO_K$.  This shows that $az$ is integral over~$\OO_K$.  Thus we have $\mathfrak{a}z\subseteq\OO_K$ and hence $z\in\fa^{-1}$.
\end{proof}

\section{Behaviour of size functions under morphisms}

Let $K$ be a number field.  Let $w=(w_0,w_1)$ and $u=(u_0,u_1)$ be two pairs of positive integers, and let $\phi\colon\P(w)_K\to\P(u)_K$ be a non-constant morphism.  Our goal in this section will be to study how the size of a point in $\P(w)(K)$ relates to the size of its image under~$\phi$.

By Lemma~\ref{lemma:morphism}, the morphism~$\phi$ is defined by a pair of non-constant homogeneous polynomials $f_0,f_1\in K[x_0,x_1]$ of degrees $eu_0$ and $eu_1$, respectively, where $e$ is the reduced degree of~$\phi$.
For $i\in\{0,1\}$, let $\fa_i$ be the fractional ideal generated by the coefficients of~$f_i$.

\begin{lemma}
\label{lemma:easy-direction}
For all $z\in K^2$, we have
$$
\II_u(f(z)) \subseteq \II_u(\fa_0,\fa_1)\II_w(z)^e.
$$
\end{lemma}

\begin{proof}
We abbreviate
$$
\fm = \II_w(z).
$$
Since $f_i$ is homogeneous of degree $eu_i$, Lemma~\ref{lemma:poly} gives
$$
f_i(z)\in \fa_i \fm^{eu_i}.
$$
It follows that
$$
\II_u(f(z)) \subseteq \II_u(\fa_0\fm^{eu_0},\fa_1\fm^{eu_1}) = \II_u(\fa_0,\fa_1)\fm^e,
$$
which proves the claim.
\end{proof}

For $i\in\{0,1\}$, we write the rational number $w_i/e$ in reduced form as
$$
\frac{w_i}{e}=\frac{\nu_i}{\delta_i}
$$
with $\nu_i,\delta_i$ coprime positive integers.

By Lemma~\ref{lemma:Nakayama}, there are integers $d_i>0$ and polynomials $g_{i,j}\in K[y_0,y_1]$ (for $i=0,1$ and $j=1,\ldots,d_i$) satisfying
\begin{equation}
x_i^{d_i}+g_{i,1}(f_0,f_1)x_i^{d_i-1}+\cdots+g_{i,d_i}(f_0,f_1)=0
\quad\text{in }K[x_0,x_1].
\label{eq:star0}
\end{equation}
After taking homogeneous components of degree~$d_iw_i$, we may and do assume that each $g_{i,j}(f_0,f_1)$ is homogeneous of degree $jw_1$.  After dividing by a power of~$x_i$ if necessary, we may and do also assume $g_{i,d_i}\ne0$.  We write
$$
g_{i,j}=\sum_{\substack{k_0,k_1\ge0\\ e(k_0u_0+k_1u_1)=jw_i}}
\gamma_{i,j,(k_0,k_1)}y_0^{k_0}y_1^{k_1}
\quad\text{with }\gamma_{i,j,(k_0,k_1)}\in K.
$$
In particular, if $g_{i,j}\ne0$, then $e$ divides $jw_i$, so $j$ is a multiple of the denominator of $w_i/e$; in other words, there is a positive integer~$l$ with $j=l\delta_i$.  Since we have ensured that $g_{i,d_i}$ is non-zero, we obtain in particular a positive integer~$m_i$ with
$$
d_i=m_i\delta_i,
$$
and all $j$ for which $g_{i,j}$ does not vanish are of the form $j=l\delta_i$ with $1\le l\le m_i$.  We can therefore rewrite \eqref{eq:star0} as
\begin{equation}
x_i^{m_i\delta_i}+\sum_{l=1}^{m_i} g_{i,l\delta_i}(f_0,f_1)x_i^{(m_i-l)\delta_i}=0
\quad\text{in }K[x_0,x_1]
\label{eq:star1}
\end{equation}
and note that
$$
g_{i,l\delta_i}=\sum_{\substack{k_0,k_1\ge0\\ k_0u_0+k_1u_1=l\nu_i}}
\gamma_{i,l\delta_j,(k_0,k_1)}y_0^{k_0}y_1^{k_1}.
$$
For $i\in\{0,1\}$ and $1\le l\le m_i$, we write $\fc_{i,l}$ for the fractional ideal generated by the coefficients of~$g_{i,l\delta_i}$, i.e.
$$
\fc_{i,l}=(\gamma_{i,l\delta_i,(k_0,k_1)}\mid k_0,k_1\ge0, k_0u_0+k_1u_1=l\nu_i).
$$
For $i\in\{0,1\}$, we write
$$
\fd_i=\II_{(1,\ldots,m_i)}(\fc_{i,},\ldots,\fc_{i,m_i}).
$$

\begin{lemma}
For all $z\in K^2$ and $i\in\{0,1\}$, we have
$$
z_i^{\delta_i}\in\fd_i\II_u(f(z))^{\nu_i}.
$$
\end{lemma}

\begin{proof}
For $i=0,1$ and $l=0,\ldots,m_i$, we write
$$
c_{i,l}=g_{i,l\delta_i}(f(z))\in K.
$$
Substituting $(x_0,x_1)=(z_0,z_1)$ in~\eqref{eq:star1}, we obtain
$$
(z_i^{\delta_i})^{m_i}+\sum_{l=1}^{m_i} c_{i,l}(z_i^{\delta_i})^{m_i-l}=0
\quad\text{for }i=0,1.
$$
We abbreviate
$$
\fm=\II_u(f(z)).
$$
Since $g_{i,l\delta_j}$ is homogeneous of degree~$l\nu_i$, Lemma~\ref{lemma:poly} gives
$$
c_{i,l}\in\fc_{i,l}\fm^{l\nu_i}.
$$
Applying Lemma~\ref{lemma:scaling-root}, we obtain
$$
z_i^{\delta_i} \in \II_{(1,\ldots,m_i)}(\fc_{i,1}\fm^{\nu_i},\ldots,\fc_{i,m_i}\fm^{m_i\nu_i})
\quad\text{for }i=0,1.
$$
This last ideal equals $\II_{(1,\ldots,m_i)}(\fc_{i,1},\ldots,\fc_{i,m_i})\fm^{\nu_i} = \fd_i\fm^{\nu_i}$.
\end{proof}

\begin{corollary}
\label{cor:hard-direction}
For all $(z_0,z_1)\in K^2$ and $i\in\{0,1\}$, we have
$$
\II_{(\nu_0,\nu_1)}(z_0^{\delta_0},z_1^{\delta_1}) \subseteq
\II_{(\nu_0,\nu_1)}(\fd_0,\fd_1) \II_u(f(z)).
$$
\end{corollary}

\begin{tm}
\label{thm:asymp}
Let $K$ be a number field, let $u$, $w$ be two pairs of positive integers, and let $\phi\colon\P(w)_K\to\P(u)_K$ be a non-constant morphism.  Let $e$ be the reduced degree of~$\phi$ (see Definition~\ref{def:reduced-degree}), and for $i=0,1$ write $w_i/e=\nu_i/\delta_i$ with $\nu_i$, $\delta_i$ coprime positive integers.  Then for all $z\in\P(w)(K)$, we have
$$
S_u(\phi(z)) \ll S_w(z)^e
$$
and
$$
S_u(\phi(z)) \gg S_{(\nu_0,\nu_1)}(z_0^{\delta_0},z_1^{\delta_1}),
$$
where the implied constants depend only on $K$, $u$, $w$ and~$\phi$.
\end{tm}

\begin{proof}
Lemma~\ref{lemma:morphism} gives us homogeneous polynomials $f_0,f_1\in K[x_0,x_1]$ such that $\phi$ is defined by $(f_0,f_1)$.
For every Archimedean place $v$ of~$K$, the set $\P(w)(K_v)$ of points of $\P(w)$ over the completion $K_v$ of~$K$ at~$v$ is in a natural way a compact topological space.  We consider the function
\begin{align*}
q_v\colon\P(w)(K_v)&\longrightarrow\R_{>0}\\
z&\longmapsto\frac{\max_{0\le i\le 1}|f_i(z)|_v^{1/u_i}}{\max_{0\le i\le 1}|z_i|_v^{e/w_i}}.
\end{align*}
Using the definitions of the size functions and the $q_v$, we compute
\begin{align*}
\frac{S_u(\phi(z))}{S_w(z)^e} &= \frac{\Norm(\II_u(f(z)))^{-1} H_{u,\infty}(f(z))}{\Norm(\II_w(z))^{-e} H_{w,\infty}(z)^e}\\
&= \Norm\bigl(\II_w(z)^e\II_u(f(z))^{-1}\bigr) \prod_{v\in\Omega_{K,\infty}} q_v(z)
\end{align*}
and
\begin{align*}
\frac{S_u(\phi(z))}{S_{(\nu_0,\nu_1)}(z_0^{\delta_0},z_1^{\delta_1})} &= \frac{\Norm(\II_u(f(z)))^{-1} H_{u,\infty}(f(z))}{\Norm(\II_{(\nu_0,\nu_1)}(z_0^{\delta_0},z_1^{\delta_1}))^{-1}H_{(\nu_0,\nu_1),\infty}(z_0^{\delta_0},z_1^{\delta_1})}\\
&= \Norm\bigl(\II_{(\nu_0,\nu_1)}(z_0^{\delta_0},z_1^{\delta_1})\II_u(f(z))^{-1}\bigr) \prod_{v\in\Omega_{K,\infty}} q_v(z).
\end{align*}
Let $\fa_i$, $\fd_i$ ($i=0,1$) be the fractional ideals defined earlier.  By Lemma~\ref{lemma:easy-direction}, we have
$$
\II_w(z)^e\II_u(f(z))^{-1}\supseteq\II_u(\fa_0,\fa_1)^{-1},
$$
and hence
$$
\Norm\bigl(\II_w(z)^e\II_u(f(z))^{-1}\bigr) \le \Norm(\II_u(\fa_0,\fa_1))^{-1}.
$$
By Corollary~\ref{cor:hard-direction}, we have
$$
\II_{(\nu_0,\nu_1)}(z_0^{\delta_0},z_1^{\delta_1})\II_u(f(z))^{-1} \subseteq \II_{(\nu_0,\nu_1)}(\fd_0,\fd_1),
$$
and hence
$$
\Norm\bigl(\II_{(\nu_0,\nu_1)}(z_0^{\delta_0},z_1^{\delta_1})\II_u(f(z))^{-1}\bigr) \ge \Norm(\II_{(\nu_0,\nu_1)}(\fd_0,\fd_1)).
$$
Finally, for each $v\in\Omega_{K,\infty}$, the function $q_v\colon\P(w)(K_v)\to\R_{>0}$ is bounded by compactness.  From this the theorem follows.
\end{proof}

\begin{corollary}
\label{cor:asymp}
In the setting of Theorem~\ref{thm:asymp}, suppose $e=1$ or $w=(1,1)$ holds.  Then for all $z\in\P(w)(K)$, we have
$$
S_u(\phi(z))\asymp S_w(z)^e,
$$
where the implied constants depend only on $K$, $u$, $w$ and~$\phi$.
\end{corollary}

\begin{proof}
First suppose $e=1$.  Then we have $\delta_i=1$ and $\nu_i=w_i$ for $i\in\{0,1\}$, and hence
$$
S_{\nu_0,\nu_1}(z_0^{\delta_0},z_1^{\delta_1}) = S_w(z) = S_w(z)^e.
$$
Next suppose $w=(1,1)$.  Then we have $\delta_i=e$ and $\nu_i=1$ for $i\in\{0,1\}$, and hence
$$
S_{(\nu_0,\nu_1)}(z_0^{\delta_0},z_1^{\delta_1}) = S_{(1,1)}(z_0^e,z_1^e) = S_{(1,1)}(z_0,z_1)^e = S_w(z)^e.
$$
In both cases, Theorem~\ref{thm:asymp} gives the result.
\end{proof}

\begin{remark}
The condition ``$e=1$ or $w=(1,1)$'' in Corollary~\ref{cor:asymp} is reminiscent of the condition ``$n=1$ or $m=1$'' in \cite[Proposition 2.1]{HS}.
\end{remark}

\begin{remark}
By Remark~\ref{remark:representable}, the assumption $e=1$ or $w=(1,1)$ implies that every morphism satisfying the conditions of Corollary~\ref{cor:asymp} is representable.  However, the conclusion of Corollary~\ref{cor:asymp} no longer holds when ``$e=1$ or $w=(1,1)$'' is weakened to ``$\phi$ is representable''.  For example, take $u=(1,3)$ and $w=(1,3)$, and consider the morphism
\begin{align*}
\phi\colon\P(1,3)&\longrightarrow\P(1,3)\\
(x_0,x_1)&\longmapsto(x_0^2,x_1^2),
\end{align*}
which has $e=2$ and is therefore representable.  For all primes~$p$, taking $x=(p,p^2)\in\P(1,3)(\Q)$, we get
\begin{align*}
S_w(x) &= S_{(1,3)}(p,p^2) = p,\\
S_u(\phi(x)) &= S_{(1,3)}(p^2,p^4) = S_{(1,3)}(p,p) = p.
\end{align*}
On the other hand, for all primes~$p$, taking $x=(1,p)\in\P(1,3)(\Q)$, we get
\begin{align*}
S_w(x) &= S_{(1,3)}(1,p) = p^{1/3},\\
S_u(\phi(x)) &= S_{(1,3)}(1,p^2) = p^{2/3}.
\end{align*}
This shows that the ratio between $S_u(\phi(x))$ and any fixed power of $S_w(x)$ is unbounded as $x$ varies.
\end{remark}

\section{Points of bounded size on modular curves}

Let $Y(1)$ be the moduli stack over $\Q$ of elliptic curves.
There is an open immersion
$$
\iota\colon Y(1)\hookrightarrow\P(4,6)_\Q
$$
defined as follows: given an elliptic curve $E$ over a $\Q$-scheme~$S$, then Zariski locally on~$S$ we can choose a non-zero differential $\omega$ and define
$$
\iota(E)=(c_4(E,\omega),c_6(E,\omega)),
$$
where $c_4$ and $c_6$ are defined in the usual way.
A different choice of~$\omega$ gives the same point of $\P(4,6)_\Q$, so $\iota$ is well defined.

\begin{definition}
\label{def:size-ell}
Let $K$ be a number field.
Using the morphism $\iota$, we define the \emph{size function}
$$
S_K\colon Y(1)(K)\longrightarrow\R_{>0}
$$
as the composition
$$
\xymatrix@=12mm{
Y(1)(K) \ar[r]^(.48){\iota(K)}& \P(4,6)(K)\ar[r]^(.58){S_{(4,6),K}}& \R_{>0}.}
$$
\end{definition}

\begin{remark}
\label{remark:size-comparison}
If $E$ is given in short Weierstrass form as
$$
E\colon y^2 = x^3 + ax + b,
$$
then we have
$$
\iota(E) = (-48a,-864b)
$$
and hence
$$
S_K(E) = S_{(4,6),K}(-48a,-864b)
\asymp \max\{|a|^{1/4},|b|^{1/6}\}.
$$
This shows that if $E$ is an elliptic curve over~$\Q$, then the ratio between $S_\Q(E)^{12}$ and the height of~$E$
$$h(E)=\max\{|a|^{3},|b|^{2}\},$$
as defined in~\cite{HS}, is bounded from above and below by a constant.
\end{remark}

Now let $n$ be a positive integer, and let $G$ be a subgroup of $\GL_2(\Z/n\Z)$.
Let $K_G$ be the subfield of the cyclotomic field $\Q(\zeta_n)$ fixed by~$G$, where $G$ acts on $\Q(\zeta_n)$ by $(g,\zeta_n)\mapsto\zeta_n^{\det g}$.
Let $Y_G$ be the moduli stack of elliptic curves with $G$-level structure, viewed as an algebraic stack over~$K_G$.
There is a canonical morphism of stacks
$$
\pi_G\colon Y_G\to Y(1)_{K_G}.
$$
Let $K$ be a finite extension of $K_G$.
We define
$$
\E_{G,K} = \{\text{elliptic curves admitting a $G$-level structure over $K$}\}/{\cong}
$$
and
$$
N_{G,K}(X) = \#\{E\in \E_{G,K}\mid S_K(E)^{12}\le X\}.
$$

\begin{lemma}
\label{lemma:conditions}
Let $n$ be a positive integer, let $G$ be a subgroup of $\GL_2(\Z/n\Z)$, and let $w$ be a pair of positive integers.  The following are equivalent:
\begin{enumerate}
  \item There is a commutative diagram
$$
\xymatrix{
Y_G \ar[r]^(.42){\iota_G} \ar[d]_{\pi_G} & \P(w)_{K_G} \ar[d]^\phi \\
Y(1)_{K_G} \ar[r]^(.46)\iota & \P(4,6)_{K_G}}
$$
of algebraic stacks over~$K_G$,
where $\iota_G$ is an open immersion and $\phi$ is representable.
  \item The integral closure of $X(1)=\P(4,6)$ in the function field of~$Y_G$ is isomorphic to $\P(w)$.
  \item The moduli space of generalized elliptic curves with $G$-level structure is isomorphic to $\P(w)$.
\end{enumerate}
\end{lemma}

\begin{proof}
The equivalence of (ii) and~(iii) follows from the fact that the integral closure from~(ii) is canonically isomorphic to the moduli space of generalized elliptic curves with $G$-level structure \cite[IV, Th\'eor\`eme 6.7(ii)]{DR}.

The implication (ii)$\;\Longrightarrow\;$(i) follows from the fact that the integral closure of $X(1)$ in the function field of~$Y_G$ fits in a commutative diagram as above.

The implication (i)$\;\Longrightarrow\;$(ii) follows from Corollary~\ref{cor:normalization} applied to $V=\iota_G(Y_G)$.
\end{proof}

\begin{remark}
\label{remark:counterexample}
If $G$ is a group satisfying the equivalent conditions of Lemma~\ref{lemma:conditions}, then the coarse moduli space of $X_G$ is isomorphic to~$\P^1$.  The converse does not hold.  For example, taking $G$ to be the group of upper-triangular matrices in $\GL_2(\Z/3\Z)$ gives the modular curve $X_G=X_0(3)$.  The coarse moduli space of $X_0(3)$ is isomorphic to~$\P^1$, but $X_0(3)$ itself (being the quotient of $X_1(3)$ by the trivial action of $\{\pm1\}$) is isomorphic to $\P(2) \times \P(1,3)$, which is not a weighted projective line.  One way to see this is to note that the Picard group of $\P(2)\times\P(1,3)$ contains a subgroup of order~2 (coming from the factor $\P(2)$), whereas the Picard group of a weighted projective line is infinite cyclic \cite[Example 7.27]{FMN}.
\end{remark}

\begin{remark}
The equivalent conditions of Lemma~\ref{lemma:conditions} hold if the graded $K_G$-algebra of modular forms for~$G$ is generated by two homogeneous elements.  Over~$\C$, the groups for which this happens were classified by Bannai, Koike, Munemasa and Sekiguchi \cite{BKMS}.
\end{remark}

\begin{tm}
\label{thm:final}
Let $n$ be a positive integer, and let $G$ be a subgroup of\/ $\GL_2(\Z/n\Z)$.
Let $K_G$ be the fixed field of the action of~$G$ on $\Q(\zeta_n)$ given by $(g,\zeta_n)\mapsto\zeta_n^{\det g}$.
Assume that $G$ satisfies the equivalent conditions of Lemma~\ref{lemma:conditions} for some $(w_0,w_1)$, and let $e(G)$ be the reduced degree of the canonical morphism $X_G\to X(1)$ (see Definition~\ref{def:reduced-degree}).
Furthermore, assume $e(G)=1$ or $w=(1,1)$ holds.
Then for every finite extension $K$ of~$K_G$, we have
$$
N_{G,K}(X) \asymp X^{1/d(G)}
\quad\text{as }X\to\infty,
$$
where
$$
d(G) = \frac{12e(G)}{w_0+w_1}.
$$
\end{tm}


\begin{proof}
Using the commutative diagram of Lemma~\ref{lemma:conditions} and noting that for counting purposes we may ignore the cusps (cf.\ \cite[Remark~6.2]{Deng}), we obtain
$$
N_{G,K}(X) \asymp \#\{z\in\P(w)(K)\mid S_{(4,6)}(\phi(z))^{12}\le X\}.
$$
By Corollary~\ref{cor:asymp} with $u=(4,6)$, the quotient $S_{(4,6)}(\phi(z))/S_w(z)^e$ is bounded.  This implies
$$
N_{G,K}(X) \asymp \#\{z\in\P(w)(K)\mid S_w(z)\le X^{1/(12e(G))}\}.
$$
Applying Theorem~\ref{thm:count}, we obtain
$$
N_{G,K}(X) \asymp X^{(w_0+w_1)/(12e(G))}.
$$
This proves the claim.
\end{proof}

\section{Examples}

The groups corresponding to the 15 torsion groups from Mazur's theorem satisfy the conditions of Lemma~\ref{lemma:conditions}.
In \Cref{table:examples}, we list these groups and a few more satisfying these conditions.

\begin{table}[t]
\begin{tabular}{|cccccc|}
  \hline
  $G$ & $\Gamma$ & $[\SL_2(\Z):\Gamma]$ & $(w_0,w_1)$ & $e(G)$ & $d(G)$ \\
  \hline
  $G_1(1)$ & $\Gamma(1)=\SL_2(\Z)$ & 1 & $(4,6)$ & 1 & 6/5 \\  
  $G_1(2)$ & $\Gamma_1(2)=\Gamma_0(2)$ & 3 & $(2,4)$ & 1 & 2 \\  
  $G_1(3)$ & $\Gamma_1(3)$ & 8 & $(1,3)$ & 1 & 3 \\  
  $G_1(4)$ & $\Gamma_1(4)$ & 12 & $(1,2)$ & 1 & 4 \\  
  $G_1(5)$ & $\Gamma_1(5)$ & 24 & $(1,1)$ & 1 & 6 \\  
  $G_1(6)$ & $\Gamma_1(6)$ & 24 & $(1,1)$ & 1 & 6 \\  
  $G_1(7)$ & $\Gamma_1(7)$ & 48 & $(1,1)$ & 2 & 12 \\
  $G_1(8)$ & $\Gamma_1(8)$ & 48 & $(1,1)$ & 2 & 12 \\
  $G_1(9)$ & $\Gamma_1(9)$ & 72 & $(1,1)$ & 3 & 18 \\
  $G_1(10)$ & $\Gamma_1(10)$ & 72 & $(1,1)$ & 3 & 18 \\
  $G_1(12)$ & $\Gamma_1(12)$ & 96 & $(1,1)$ & 4 & 24 \\
  $G(2,2)$ & $\Gamma(2)$ & 6 & $(2,2)$ & 1 & 3 \\  
  $G(2,4)$ & $\Gamma(2,4)$ & 24 & $(1,1)$ & 1 & 6 \\  
  $G(2,6)$ & $\Gamma(2,6)$ & 48 & $(1,1)$ & 2 & 12 \\
  $G(2,8)$ & $\Gamma(2,8)$ & 96 & $(1,1)$ & 4 & 24 \\
  \hline
  $G_0(4)$ & $\Gamma_0(4)$ & 6 & $(2,2)$ & 1 & 3 \\  
  $G(4,4)$ & $\Gamma(4)$ & 48 & $(1,1)$ & 2 & 12 \\
  $G_0(8)\cap G_1(4)$ & $\Gamma_0(8)\cap\Gamma_1(4)$ & 24 & $(1,1)$ & 1 & 6 \\  
  $G(3,3)$ & $\Gamma(3)$ & 24 & $(1,1)$ & 1 &6 \\  
  $G(3,6)$ & $\Gamma(3,6)$ & 72 & $(1,1)$ & 3 & 18 \\
  $G_0(9)\cap G_1(3)$ & $\Gamma_0(9)\cap\Gamma_1(3)$ & 24 & $(1,1)$ & 1 & 6 \\  
  $G(5,5)$ & $\Gamma(5)$ & 120 & $(1,1)$ & 5 & 30 \\
  \hline
\end{tabular}
\medskip
\caption{A selection of groups satisfying the conditions of Lemma~\ref{lemma:conditions}.  The first 15 groups are those appearing in Mazur's theorem.}
\label{table:examples}
\end{table}

For positive integers $m\mid n$ we write
$$
G(m,n) = \biggl\{ g\in\GL_2(\Z/n\Z)\biggm|
g = \biggl(\begin{matrix}*&*\\0&1\end{matrix}\biggr)
\text{ and }
g\equiv\biggl(\begin{matrix}*&0\\0&1\end{matrix}\biggr)\pmod m \biggr\}.
$$
We also put
$$
G_1(n) = G(1,n)
$$
and
$$
G_0(n) = \biggl\{ g\in\GL_2(\Z/n\Z)\biggm|
g = \biggl(\begin{matrix}*&*\\0&*\end{matrix}\biggl)\biggr\}.
$$
For each group $G$ we give its inverse image $\Gamma$ under the canonical group homomorphism $\SL_2(\Z)\to\GL_2(\Z/n\Z)$, the index of $\Gamma$ in $\SL_2(\Z)$, the weights of the corresponding weighted projective line, and the values $e(G)$ and~$d(G)$.
The first 12 groups can also be found in \cite[Examples 2.1 and Example 2.5]{Meier}, and the 12 groups with $e(G)=1$ can also be found in \cite[Table~1]{BKMS}.
By construction, for all groups~$G$ in the table, the determinant $G\to(\Z/n\Z)^\times$ is surjective, hence the index $[\GL_2(\Z/n\Z):G]$ equals $[\SL_2(\Z):\Gamma]$, and $K_G$ equals $\Q$.
Furthermore, we note that the numbers $e(G)$ and~$d(G)$ can be expressed as
$$
\begin{aligned}
e(G) &= \frac{w_0w_1}{24}[\SL_2(\Z):\Gamma],\\
d(G) &= \frac{w_0w_1}{2(w_0+w_1)}[\SL_2(\Z):\Gamma].
\end{aligned}
$$



\section{Future work}

In work in progress of I. Manterola Ayala and the first author (see \cite{Irati}), results will be proved that make it possible to count points of a moduli stack of the form $\P(w)$ directly with respect to the pull-back of the size function from $X(1)$, rather than first relating this pull-back to the standard size function on~$\P(w)$.  This approach requires extending the work of Deng~\cite{Deng}, but is conceptually simpler than the approach we have taken here.

It would be interesting to obtain a result similar to Theorem~\ref{thm:final} for moduli stacks of elliptic curves that are of the form $\P(2) \times \P(1,1)$. An example of such a moduli stack is $X_0(6)$, so such a result would enable one to count elliptic curves with a 6-isogeny over any number field.

\section{Acknowledgements}
We are grateful to Pieter Moree for his effort in organizing our collaboration at MPIM Bonn, at which a large part of the work leading to this paper was conducted.
We would also like to thank Soumya Sankar and Brandon Boggess for sending us an early version of their manuscript \cite{Boggess-Sankar}.
Finally, we thank the referee for several useful comments, which in particular led to simpler proofs of Lemma~\ref{lemma:finite} and Lemma~\ref{lemma:morphism-final}.

\appendix

\section{Morphisms between quotient stacks}

In this appendix we assume some knowledge of stacks.  We place ourselves in the following situation.  Let $S$ be a scheme, let $G$ and~$H$ be two group schemes over~$S$, and let $m_G\colon G\times_S G\to G$ and $m_H\colon H\times_S H\to H$ be the group operations.  Let $X$ and~$Y$ be two $S$-schemes, let $a\colon G\times_S X\to X$ be a left action of $G$ on~$X$, and let $b\colon H\times_S Y\to Y$ be a left action of $H$ on~$Y$.  Let $p_2\colon G\times_S X\to X$ be the second projection, and let $p_{2,3}\colon G\times_S G\times_S X\to G\times_S X$ be the projection onto the second and third factors.

We consider the quotient stacks $[G\backslash X]$ and $[H\backslash Y]$ over (the \emph{fppf} site of) $S$, writing quotients on the left because $a$ and~$b$ are left actions.  Below we give an explicit description of the groupoid of morphisms $[G\backslash X]\to[H\backslash Y]$ of stacks over~$S$.
For this we will use the following description of morphisms from the quotient stack $[G\backslash X]$ to another stack~$\mathcal{Y}$ given by Noohi \cite[Proposition~3.19]{Noohi}; see also \citestacks{044U} for part of this statement.

\begin{lemma}
\label{lemma:morphism-general}
Let $\mathcal{Y}$ be a stack in groupoids over~$S$, and let $C([G\backslash X],\mathcal{Y})$ be the following groupoid.  The objects are the pairs $(f,h)$ where $f\colon X\to\mathcal{Y}$ is a morphism of stacks and $h$ is a descent datum for\/~$f$, i.e.\ an isomorphism $h\colon f\circ p_2\isom f\circ a$ of functors $G\times_S X\to\mathcal{Y}$ satisfying
$$
(m_G\times\id_X)^* h = ({\id_G}\times a)^* h \circ p_{2,3}^* h.
$$
The morphisms from $(f,h)$ to $(f',h')$ are the isomorphisms $c\colon f\isom f'$ of functors $X\to\mathcal{Y}$ satisfying
$$
a^*c\circ h = h'\circ p_2^*c.
$$
Then the groupoid of morphisms $[G\backslash X]\to\mathcal{Y}$ is canonically equivalent to $C([G\backslash X],\mathcal{Y})$.
\end{lemma}

To state the next lemma, we recall the following.  Given a left action of a group~$\Gamma$ on a set~$Z$, the \emph{quotient groupoid} $\Gamma\doublebackslash Z$ is the following groupoid: the set of objects is $Z$, the morphisms $z\to z'$ are the elements $\gamma\in\Gamma$ with $\gamma z=z'$, and composition of morphisms is the group operation in~$\Gamma$.  The set of isomorphism classes of $\Gamma\doublebackslash Z$ is just the quotient set $\Gamma\backslash Z$.

\begin{lemma}
\label{lemma:morphism-final}
In the above situation, assume in addition that all $H$-torsors on~$X$ are trivial.  Let $Z$ be the set of pairs $(f\colon X\to Y,h\colon G\times_S X\to H)$ of morphisms of $S$-schemes such that for all $S$-schemes $T$, all $x\in X(T)$ and all $g,g'\in G(T)$ we have
\begin{equation}
h(g'g,x)=h(g',gx)h(g,x)
\label{eq:cocycle}
\end{equation}
and
\begin{equation}
f(a(g,x)) = b(h(g,x),f(x)).
\label{eq:compat}
\end{equation}
Let the group $H(X)$ act on $Z$ by
$$
(c,(f,h))\mapsto(f',h'),
$$
where $f'$ and~$h'$ are defined on points as follows: for all $S$-schemes $T$, all $x\in X(T)$ and all $g\in G(T)$ we have
$$
f'(x) = b(c(x),f(x))
$$
and
$$
h'(g,x) = c(a(g,x))h(g,x)c(x)^{-1}.
$$
Then the groupoid of morphisms $[G\backslash X]\to[H\backslash Y]$ is canonically equivalent to the quotient groupoid $H(X)\doublebackslash Z$.  In particular, there is a canonical bijection between the set of isomorphism classes of such morphisms and the quotient set $H(X)\backslash Z$.
\end{lemma}

\begin{proof}
We apply Lemma~\ref{lemma:morphism-general} with $\mathcal{Y}=[H\backslash Y]$.  Because all $H$-torsors on~$X$ are trivial, the groupoid of morphisms $X\to[H\backslash Y]$ is canonically equivalent to the groupoid $D(X,[H\backslash Y])$ defined as follows: the objects of $D(X,[H\backslash Y])$ are the morphisms $f\colon X\to Y$ of schemes, and the isomorphisms $f\isom f'$ in $D(X,[H\backslash Y])$ are the elements $c\in H(X)$ such that the diagram
$$
\xymatrix{
X \ar[d]_{(c,f)} \ar[r]^{f'}& Y\\
H\times_S Y \ar[ur]_b}
$$
is commutative.  Similarly, the isomorphisms $f\circ p_2\isom f\circ a$ in the groupoid of morphisms $G\times_S X\to[H\backslash Y]$ correspond to the elements $h\in H(G\times_S X)$ such that the diagram
$$
\xymatrix{
G\times_S X \ar[r]^(.6){f\circ a} \ar[d]_{(h,f\circ p_2)} & X\\
H\times_S Y \ar[ur]_b}
$$
is commutative.  Furthermore, such an $h$ is a descent datum for~$f$ if and only if the diagram
$$
\xymatrix{
G\times_S G\times_S X \ar[r]^(.56){m_G\times\id_X} \ar[d]_{({\id_G}\times a,p_{2,3})} & G\times_S X \ar[r]^(.6)h & H\\
(G\times_S X)\times_S(G\times_S X) \ar[r]^(.66){h\times h}& H\times_S H \ar[ru]_{m_H}}
$$
is commutative.
On $T$-valued points, the commutativity of the last two diagrams comes down to \eqref{eq:compat} and~\eqref{eq:cocycle}, respectively, so the objects of $C([G\backslash X],[H\backslash Y])$ correspond to the elements of~$Z$.
The isomorphisms $(f,h)\isom(f',h')$ in $C([G\backslash X],[H\backslash Y])$ correspond to the elements $c\in H(X)$ as above such that in addition the diagram
$$
\xymatrix{
G\times_S X \ar[r]^{(c\circ a,h)} \ar[d]_{(h',c\circ p_2)} & H\times_S H \ar[d]^{m_H}\\
H\times_S H \ar[r]^(.6){m_H} & H}
$$
is commutative.  Equivalently, these isomorphisms correspond to the elements $c\in H(X)$ sending $(f,h)$ to $(f',h')$ under the given action of $H(X)$ on~$Z$.
\end{proof}


\end{document}